\documentclass[12pt, reqno]{amsart}
\usepackage{amsmath, amsthm, amscd, amsfonts, amssymb, graphicx, color}
\usepackage[english]{babel}
\usepackage[bookmarksnumbered, colorlinks, plainpages]{hyperref}
\hypersetup{colorlinks=true,linkcolor=red, anchorcolor=green, citecolor=cyan, urlcolor=red, filecolor=magenta, pdftoolbar=true}

\newtheorem{theorem}{Theorem}[section]
\newtheorem{lemma}[theorem]{Lemma}

\theoremstyle{definition}
\newtheorem{definition}[theorem]{Definition}

\newtheorem{problem}[theorem]{Problem}

\theoremstyle{remark}

\numberwithin{equation}{section}

\begin{document}

\setcounter{page}{1}

\title[Solvability of Mixed problems for a Space-time Degenerating ...]{Solvability of a Mixed Problem for a Time-Fractional PDE with Time-Space Degenerating Coefficients}

\author[B.H. Toshtemirov]{Bakhodirjon Toshtemirov$^{a}$}

\author[A.O. Mamanazarov]{Azizbek Mamanazarov$^{a}$}

\let\thefootnote\relax
\footnote{ $^{a}$ Department of General Sciences, Alasala Colleges, Saudi Arabia and Kokand University, Uzbekistan \\ Email: bakhodirjon.toshtemirov@alasala.edu.sa \& toshtemirovbh@gmail.com
\\ 
$^{b}$ Department of Mathematical Analysis and Differential Equations, Fergana State University, Uzbekistan\\ Email: mamanazarovaz1992@gmail.com}

\thanks{All authors contributed equally to the manuscript and read and approved the final manuscript.}

\let\thefootnote\relax\footnote{$^{a}$Corresponding author}



\begin{abstract} 

In this paper, we investigate the unique solvability of a mixed boundary value problem for a fractional partial differential equation featuring a degenerate coefficient. By introducing a novel operator and applying the method of separation of variables, we establish the existence of eigenvalues and eigenfunctions for the associated spectral problem and prove that the operator possesses a discrete spectrum. Additionally, we establish the relationship between the given data and the unique solvability of the problem, offering new insights into how degeneracy influences fractional diffusion processes.

\end{abstract} 
\maketitle


\section{Introduction}

Learning fractional partial differential equations has been gaining huge interest from scientists due to its huge application in physics, engineering, and other scientific fields \cite{Uchai}. 
The lack of a single definition of the fractional order differential operators (FDO) and the reason each FDO has its role and application motivates the usage of newly introduced FDOs in different classic problems in understanding their effects and applications.
Another motivation for applying FDO is the generalization of other FDOs. For example, Hilfer FDO \cite{Hil1} generalizes Caputo and Riemann-Liouville FDOs and the properties of this operator confirm the properties of both operators above under some conditions, for further works involving Hilfer or bi-ordinal Hilfer FDOs see the refrences \cite{my5}.

In 1966, I. Dimovski \cite{Dim}, earlier in 1966, introduced the hyper-Bessel
differential operator of higher (integer) order $m\geq 1$:
\begin{equation}\label{eq:1.9}
	B:=t^{\alpha_0}\frac{d}{dt}t^{\alpha_1}\frac{d}{dt}t^{\alpha_2} ... \frac{d}{dt}t^{\alpha_{m-1}}\frac{d}{dt}t^{\alpha_m},    ~ ~ t>0
	\end{equation}
with $\beta=m-(\alpha_0+\alpha_1+...+\alpha_m)>0$.

Later, in \cite{Kir}, as well as, Kiryakova  presented fractional multi-order analogues of the hyper-Bessel operators (\ref{eq:1.9}) identified as particular cases of the generalized fractional derivatives of multi-order $\delta=(\delta_1,...,\delta_m)$
\begin{equation*}
\mathcal{D}=t^{\alpha_0}\left(\frac{d}{dt}\right)^{\delta_1}t^{\alpha_1}\left(\frac{d}{dt}\right)^{\delta_2}t^{\alpha_2} ... \left(\frac{d}{dt}\right)^{\delta_{m-1}}t^{\alpha_{m-1}}\left(\frac{d}{dt}\right)^{\delta_m}t^{\alpha_m}
\end{equation*}

After that Roberto Garra et.al.\cite{Garra2} considered a particular operator that is suitable to generalize the standard process of relaxation considering both memory effects of power-law type and time variability of the characteristic coefficient. They proposed the particular hyper-Bessel operator of order $0<\alpha\leq 1$ identified in terms of Erdelyi-Kober (E-K integral or derivative can be represented as follows
 	 	\begin{equation}\label{eq:1.10}
 	\left(t^{\theta}\frac{d}{dt}\right)^{\alpha}f(t)=\left\{\begin{gathered}
 		(1-\theta)^{\alpha}t^{-\alpha(1-\theta)}I_{1-\theta}^{0, -\alpha} f(t), ~ ~ \text{if} ~ ~ \theta<1,
 		\hfill \cr
 			(\theta-1)^{\alpha}I_{1-\theta}^{-1, -\alpha} t^{(1-\theta)\alpha}f(t), ~ ~ \text{if} ~ ~ \theta>1.   	
 	\end{gathered}\right.
  	\end{equation}

The operator (\ref{eq:1.10})  coincides with the Riemann-Liouville fractional derivative when $\theta=0$.

When $\theta=1$ the explicit form of the operator $\displaystyle\left(t\frac{d}{dt}\right)^{\alpha}$ comes from the theory
of fractional powers of operators such that
\begin{equation}
	\left(t\frac{d}{dt}\right)^{\alpha}f(t)=\delta\left(\mathcal{J}^{1-\alpha}_{t_{0}^{+}}f\right)(t), ~ ~ ~ 0<\alpha<1,
\end{equation} 
where $\displaystyle\delta=t\frac{d}{dt}$ and 
\begin{equation*}
\left(\mathcal{J}^{1-\alpha}_{t_{0}^{+}}f\right)(t)=\frac{1}{\Gamma(1-\alpha)}\int\limits_{t_0}^{t}\left(\ln \frac{x}{u} \right)^{-\alpha} f(u) \frac{du}{u}, ~ ~ t_0\geq 0
\end{equation*}
is the Hadamard fractional integral of order $1-\alpha$.

The regularized Caputo-like counterpart of the hyper-Bessel fractional differential operator has been also introduced and used to model the relaxation process in the same reference. The possibility of the usage of this FDO as two different operators is one of the motivations for us to study the following differential equation.

Let $\Omega$ denote a rectangular domain in the $xOt$ plane, specified as
\[
\Omega = \{(x, t) \mid 0 < x < 1, \quad a < t \leq T < \infty\},
\]
where $T$ is a given positive real constant. In the domain $\Omega$, we shall examine the following equation: 
\begin{equation}\label{eq3.1}
{ }^C\Big(\big(t-a)^\theta \frac{\partial}{\partial t} \Big)^\alpha  u(x,t)-\frac{\partial}{\partial x}\left(x^{\beta}\frac{\partial u(x,t)}{\partial x}\right)=f\left(x,t\right),
\end{equation}
where $\alpha \in (0, 1)$, $\theta<1$, $\beta \in (0,2)$, $\beta \ne 1$, $f(x,t)$ is a given source term. 

${ }^C D_{a+,t}^\alpha$ is the regularized Caputo-like counterpart of the
hyper-Bessel operator with arbitrary starting point $a$. 
The time-fractional derivative introduces a power-law memory effect, capturing relaxation behaviors where past states significantly influence future outcomes.  The key role of the Caputo-like counterpart of the hyper-Bessel fractional-order derivative is its ability to model processes with long-term dependencies, which are often observed in materials and systems with anomalous diffusion (see \cite{hb1}, \cite{my2}).

Moreover, the spatial term in the equation reflects a diffusion process influenced by spatially variable coefficients, as is commonly observed in inhomogeneous media. The $x^\beta$ coefficient means that diffusion is faster or slower depending on the spatial location. 
  Our approach showed that the degree of degeneration 
$\beta$ affects both the formulation and investigation of the problem. This equation may be applied to model anomalous diffusion in porous media, heat conduction in materials with variable conductivity, and biological transport processes, where the diffusion characteristics depend on both time and spatial location. For $0<\beta<1$ the variable-coefficient diffusion equation admits classical solutions, while for $1<\beta<2$ only weak solutions exist, making it particularly useful for modeling phenomena such as heat conduction in heterogeneous porous media where diffusivity varies spatially.  Furthermore, we may refer to some references  \cite{Urinov},  \cite{Maciag} \cite{Khabiri} that similar spatial terms have been used to study degenerate beam vibration equations applying different methods.
   
We will formulate and study an initial boundary value problem for different degrees of degeneration in equation \eqref{eq3.1}. The rest of the paper is organized as follows: In Section 2, we recall auxiliary results about Mittag-Leffler function, the Caputo-like counterpart of the hyper-Bessel fractional differential operator. In Section 3, we explore how the formulation of boundary conditions is influenced by the extent of degeneracy in the diffusion operator. As the degree of degeneration increases, the boundary conditions must be adapted accordingly to ensure the well-posedness of the solution. In Section 4, the main problem formulation is given. Section 5 discusses the spectral problem and we will investigate the order of Fourier coefficients in Section 6. In Section 7, we will analyse the existence of the solution and last section is devoted to studying the uniqueness of the solution.

\section{Preliminaries}
In this section, we present the necessary definitions and properties of the Mittag-Leffler function, the fractional differential operators and some function spaces that will be utilized throughout the paper.

\subsection{Mittag-Leffler Function}
The two-parameter Mittag-Leffler (M-L) function is an entire function and given by
\begin{equation*}\label{e1}
	E_{\alpha, \beta}(z)=\sum_{k=0}^{+\infty}\frac{z^k}{\Gamma(\alpha{k}+\beta)}, \, \, \, \alpha>0, \, \beta\in\mathbb{R}.
\end{equation*}

\begin{lemma}\label{lem: 2.1} (see \cite{Podl}) Let $\alpha<2, \, \beta\in\mathbb{R}$ and $\frac{\pi\alpha}{2}<\mu<\min\{\pi, \pi\alpha\}$. Then the following estimate holds true
$$
|E_{\alpha, \beta}(z)|\leq\frac{M}{1+|z|}, \, \, \, \mu\leq|argz|\leq\pi, \, \, |z|\geq0.
$$\end{lemma}
{Here $M$ is a positive constant.}
 
\begin{lemma}\label{Lem: 1.1.4}  \cite{Kil} If $\alpha>{0}$  and $\beta\in\mathbb{C}$, then the following recurrence formula holds:
$$
E_{\alpha, \beta}(z)=\frac{1}{\Gamma(\beta)}+zE_{\alpha, \alpha+\beta}(z).
$$
\end{lemma}

\subsection{Caputo-like counterpart of the hyper-Bessel FDO with arbitrary starting point}

 \label{Def:1.3.17} A regularized Caputo-like counterpart of the operator (\ref{eq:1.10}) is defined for $\theta<1$  in terms of E-K fractional order operator  such that 
	\begin{equation}\label{1.12}
^C\left(t^{\theta}\frac{d}{dt}\right)^{\alpha}f(t)=(1-\theta)^{\alpha}t^{-\alpha(1-\theta)}I_{1-\theta}^{0, -\alpha} [f(t)-f(0)]
	\end{equation}
or in terms of the hyper-Bessel differential operator
\begin{equation}
	^C\left(t^{\theta}\frac{d}{dt}\right)^{\alpha}f(t)=\left(t^{\theta}\frac{d}{dt}\right)^{\alpha}f(t)-\frac{f(0)t^{-\alpha(1-\theta)}}{(1-\theta)^{-\alpha}\Gamma(1-\alpha)},
\end{equation}
where $f(0)$ is the initial condition.

The regularized Caputo-like counterpart of the
hyper-Bessel operator can be also defined with arbitrary starting point $a$, see \cite{my1}.
\begin{definition}\label{Def:3.1.1}
	Regularized Caputo-like counterpart of the hyper-Bessel fractional differential operator for  $\theta<1$,  $0<\alpha\leq1$ and $t>a\geq{0}$ is defined 
	\begin{equation}\label{3.3}
		^{C}\left((t^{\theta}-a^\theta)\frac{d}{dt}\right)^\alpha{f}(t)=\Big((t^{\theta}-a^{\theta})\frac{d}{dt}\Big)^\alpha{f}(t)-\frac{f(a)\Big(t^{(1-\theta)}-a^{(1-\theta)}\Big)^{-\alpha}}{(1-\theta)^{-\alpha}\Gamma(1-\alpha)},
	\end{equation}
	where
	$$
	\left((t^\theta-a^\theta)\frac{d}{dt}\right)^\alpha{f}(t)=\left\{ \begin{gathered}
		(1-\theta)^\alpha t^{-(1-\theta)\alpha}I_{1-\theta, a+}^{0, -\alpha} f(t) \, \, \,  \textrm{if}  \, \,  \,  \theta<1, \hfill \cr
		(\theta-1)^\alpha t^{-(1-\theta)\alpha}I_{1-\theta, a+}^{-1, -\alpha} f(t)  \, \, \,  \textrm{if} \, \,  \,  \theta>1,  \hfill \cr
	\end{gathered}  \right.
	$$
	is a hyper-Bessel fractional differential operator.
	\end{definition}

We note that $I_{\beta}^{\gamma, -\alpha}:=D_{\beta}^{\gamma-\alpha, \alpha}$  is the interpretation of Erdelyi-Kober (E-K) integral for negative order such that
\begin{equation*}
I_{\beta}^{\gamma, -\alpha}:=(\gamma-\alpha+1)I_{\beta}^{\gamma, 1-\alpha}f(t)+\frac{1}{\beta}I_{\beta}^{\gamma, 1-\alpha}\left(t\frac{d}{dt}f(t)\right).
\end{equation*}

	The Erdelyi-Kober (E-K) fractional integral of a function  $f(t)\in{C_\mu}$  $\mu \geq -\beta(\gamma+1) $ with arbitrary parameters $\delta>0, \gamma\in{\mathbb{R}}$ and $\beta>0$ is defined as (\cite{Kir})
	$$
	I^{\gamma, \delta}_{\beta; a+}f(t)=\frac{t^{-\beta(\gamma+\delta)}}{\Gamma(\delta)}\int\limits_{a}^{t}(t^\beta-\tau^\beta)^{\delta-1}\tau^{\beta\gamma}f(\tau)d(\tau^\beta),
	$$
	which can be reduced up to a weight to  $I^{q}_{a+}f(t)$ (Riemann-Liouville fractional integral) at $\gamma=0$ and $\beta=1$, and Erdelyi-Kober fractional derivative of $f(t)\in{C}_{\mu}^{(n)}$ for $n-1<\delta\leq{n}, n\in\mathbb{N}$ is defined by
	$$
	D_{\beta, a+}^{\gamma, \delta}f(t)=\prod_{j=1}^{n}\left(\gamma+j+\frac{t}{\beta}\frac{d}{dt}\right)\big(I_{\beta, a+}^{\gamma+\delta, n-\delta} f(t)\big),
	$$
where $C_{\mu}^{(n)}$  is the  weighted space of continuous functions defined as follows
\begin{equation}\label{C1.2.1}
		C_{\mu}^{(n)}:=\left\{f(t)=t^pf_1(t), ~ ~ f_1(t)\in C^{(n)}[0, \infty)\right\}, ~ ~   C_{\mu} :=C_{\mu}^{(0)} ~  \text{with} ~ \nu\in\mathbb{R}
	\end{equation}
	if there exists $p>\mu$ for fixed $\mu\geq -1$.

\begin{theorem}\label{theo2.4}  If $f(t)\in C_{\mu}[a, \infty)$, then the non-homogeneous fractional differential equation for $~\alpha\in(0, 1), ~ \theta<1$
	\begin{equation}\label{equat2.4}
		^C\left(\big(t^{\theta}-a^\theta\big)\frac{d}{dt}\right)^{\alpha}u(t)+\lambda u(t)=f(t), ~ 
	\end{equation} 
	satisfying the initial condition  $u(a^+)=u_a$, has a unique solution represented by
	\begin{equation}\label{eq2.6}
		\begin{gathered}
		u(t)=u_a E_{\alpha, 1}\Big(\lambda^*\big(t^p-a^p\big)^\alpha\Big)+B(t),
		\end{gathered}
	\end{equation}
 where $p=1-\theta$ and
	$$
	\begin{array}{l}
		\displaystyle{B(t)=\frac{1}{p^\alpha\Gamma(\alpha)}\int_{a}^{t}\big(t^p-\tau^p\big)^{\alpha-1}f(\tau)d(\tau^p)}\\
		\\
		\displaystyle{+\frac{\lambda^*}{p^\alpha}\int\limits_{0}^{t}\left(t^{p}-x^{p}\right)^{2\alpha-1}E_{\alpha, 2\alpha}\left[\lambda^*\left(t^{p}-x^{p}\right)^\alpha\right]f(x)d(x^p),}
	\end{array}
	$$
	where $p=1-\theta$ and $\lambda^*=-\frac{\lambda}{p^{\alpha}}$.
	\end{theorem}

Note that a similar theorem was proved in \cite{fatm} for zero starting point.

\begin{proof} 
    It is obvious that the solution (\ref{eq2.6}) satisfies the initial condition.

    Below,  we explain the derivation of the  $\displaystyle^{C}\left((t^{\theta}-a^\theta)\frac{\partial}{\partial{t}}\right)^\alpha{u}(t)$ in (\ref{eq2.6}). By using  relation (\ref{3.3}) we get:
\begin{multline*}
^{C}\left((t^{\theta}-a^\theta)\frac{\partial}{\partial{t}}\right)^\alpha{u}(t)=\left(t^\theta\frac{\partial}{\partial{t}}\right)^\alpha\left[u_0 E_{\alpha, 1}\Big(\lambda^*\big(t^p-a^p\big)^\alpha\Big)+B(t)\right]-\\
-\frac{u_0 (t^p-a^p)^\alpha}{p^{-\alpha}\Gamma(1-\alpha)}.~ ~ ~ ~ ~ ~ 
\end{multline*}

The hyper-Bessel derivative of the Mittag-Leffler function is
\begin{equation*}
\left((t^\theta-a^\theta)\frac{\partial}{\partial{t}}\right)^\alpha u_0 E_{\alpha,1}\left(-\frac{\lambda}{p^\alpha}(t^p-a^p)^\alpha\right)=
u_0 p^\alpha(t^p-a^p)^{-\alpha}E_{\alpha, 1-\alpha}\big[-\lambda(t^p-a^p)^\alpha\big].
\end{equation*}

By means of Lemma \ref{Lem: 1.1.4}, the last expression can be also written as 
\begin{multline*}
\Big((t^\theta-a^\theta)\frac{\partial}{\partial{t}}\Big)^\alpha u_0 E_{\alpha,1}\Big[-\frac{\lambda}{p^\alpha}(t^p-a^p)^\alpha\Big]=
\frac{u_0 p^\alpha(t^p-a^p)^{-\alpha}}{\Gamma(1-\alpha)}-u_0 \lambda E_{\alpha, 1}\big[-\frac{\lambda}{p^\alpha}(t^p-a^p)^\alpha\big].
\end{multline*}

Then evaluating $\Big((t^\theta-a^\theta)\frac{\partial}{\partial{t}}\Big)^\alpha{B}(t)$ gives that
\begin{multline*}
\Big((t^\theta-a^\theta)\frac{\partial}{\partial{t}}\Big)^\alpha{B}(t)=\Big((t^\theta-a^\theta)\frac{\partial}{\partial{t}}\Big)^\alpha\Big(f^*(t)+\lambda^*\int_{a}^{t}(t^p-a^p)^{\alpha-1}E_{\alpha, \alpha}\big[\lambda^*(t^p-a^p)\big]f^*(\tau)d(\tau^p)\Big)=\\
p^\alpha{t}^{-p\alpha}D_{p, a+}^{-\alpha, \alpha}\Big(\frac{1}{p^\alpha}I_{p, a+}^{-\alpha, \alpha}t^{p\alpha}{f}(t)+\lambda^*\int_{a}^{t}(t^p-a^p)^{\alpha-1}E_{\alpha, \alpha}\big[\lambda^*(t^p-a^p)\big]f^*(\tau)d(\tau^p)\Big)=\\
f(t)+p^{\alpha}t^{-p\alpha}D_{p, a+}^{-\alpha, \alpha}\Big(\lambda^*\int_{a}^{t}(t^p-a^p)^{\alpha-1}E_{\alpha, \alpha}\big[\lambda^*(t^p-a^p)\big]f^*(\tau)d(\tau^p)\Big),
\end{multline*}
where $\lambda^*=-\frac{\lambda}{p^\alpha}$ and $$f^*(t)=\frac{1}{p^\alpha\Gamma(\alpha)}\int_{a}^{t}(t^p-\tau^p)^{\alpha-1}f(\tau)d(\tau^p).$$

The second term in the last expression can be simplified using the Erd'elyi-Kober fractional derivative for $n=1$,
\begin{multline*}
-\lambda t^{-p\alpha}\left(1-\alpha+\frac{t}{p}\frac{d}{dt}\right)\frac{t^{-p(1-\alpha)}}{\Gamma(1-\alpha)}
\int_{a}^{t}(t^p-\tau^p)^{-\alpha}d(\tau^p)\times\\
\int_{a}^{\tau}(\tau^p-s^p)^{\alpha-1}E_{\alpha, \alpha}\big[\lambda^*(\tau^p-s^p)^\alpha\big]f^*(s)d(s^p)=
\end{multline*}
\begin{multline*}
-\lambda t^{-p\alpha}\left(1-\alpha+\frac{t}{p}\frac{d}{dt}\right)\frac{t^{-p(1-\alpha)}}{\Gamma(1-\alpha)}\int_{a}^{t}f^*(s)d(s^p)\times\\
\int_{s}^{t}(t^p-\tau^p)^{-\alpha}(\tau^p-s^p)^{\alpha-1}E_{\alpha, \alpha}\big[\lambda^*(\tau^p-s^p)^\alpha\big]d(\tau^p)=
\end{multline*}
$$
-\lambda t^{-p\alpha}\left(1-\alpha+\frac{t}{p}\frac{d}{dt}\right)t^{-p(1-\alpha)}\int_{a}^{t}E_{\alpha, 1}\left[\lambda^*(t^p-s^p)^{\alpha}\right]f^*(s)d(s^p)=
$$
$$
-\lambda (1-\alpha)t^{-p}\int_{a}^{t}E_{\alpha, 1}\left[\lambda^*(t^p-s^p)^{\alpha}\right]f^*(s)d(s^p)-
$$
$$
-\frac{\lambda t^{-p\alpha+1}}{p}\frac{d}{dt}\left(t^{-p(1-\alpha)}\int_{a}^{t}E_{\alpha, 1}\left[\lambda^*(t^p-s^p)^{\alpha}\right]f^*(s)d(s^p)\right)=
$$
$$
=-\lambda (1-\alpha)t^{-p}\int_{a}^{t}E_{\alpha, 1}\left[\lambda^*(t^p-s^p)^{\alpha}\right]f^*(s)d(s^p)+
$$
$$
+\lambda (1-\alpha)t^{-p}\int_{a}^{t}E_{\alpha, 1}\left[\lambda^*(t^p-\tau^p)^{\alpha}\right]f^*(\tau)d(\tau^p)-\lambda_kf^*(t)-
$$
$$
-\lambda t^{1-p}\int_{a}^{t}\lambda^*(t^p-\tau^p)^{\alpha-1}E_{\alpha, \alpha}\left[\lambda^*(t^p-\tau^p)^{\alpha}\right]f^*(\tau)d(\tau^p)=
$$
$$
=-\lambda_k\left(f^*(t)+\lambda^*\int_{a}^{t}(t^p-a^p)^{\alpha-1}E_{\alpha, \alpha}\big[\lambda^*(t^p-a^p)\big]f^*(\tau)d(\tau^p)\right)=-\lambda B(t).
$$
Hence, we get
\begin{equation*}
^{C}\left((t^{\theta}-a^\theta)\frac{\partial}{\partial{t}}\right)^\alpha{u}(t)=
-\lambda \left[u_0 E_{\alpha, 1}\left(\frac{(k\pi)^2}{p^\alpha}(t^p-a^p)\right)+B(t)\right]+f(t).
\end{equation*}
This proves that the solution (\ref{eq2.6}) satisfies the equation
$$
^C\Big((t^\theta-a^\theta)\frac{\partial}{\partial{t}}\Big)^{\alpha}u(t)+\lambda u(t)=f(t).
$$

\end{proof}

\begin{lemma}\label{lemmacha}
     Consider the problem for (\ref{equat2.4}) with $u(a^+)=u_a$ condition. Another explicit form of the solution can also be written as follows 
     $$
u(t)=u_a E_{\alpha, 1}\left(\lambda^* (t^p-a^p)^{\alpha}\right)+\frac{1}{p^\alpha} \int_a^t\left(t^p-s^p\right)^{\alpha-1} E_{\alpha, \alpha}\left(\lambda^*\left(t^\rho-s^p\right)^\alpha\right) f(s) d\left(s^\rho\right).
$$
\end{lemma}
In \cite{Zang}, a similar theorem was proven for the same solution with a zero initial condition, providing a framework for the current analysis.

\subsection{Relevant function spaces}
 
 Let $t \in[0, T]$ and that, for every $t$, or at least for a.e. $t$, the function $u(\cdot, t)$ belongs to a separable Hilbert space $V\left(\right.$ e.g. $L^2(\Omega)$ or $\left.H^1(\Omega)\right)$.
Then, one can consider $u$ as a function of the real variable $t$ with values into $V$ :
$$
u:[0, T] \rightarrow V.
$$

The set $C([0, T]; V)$ of  continuous functions $u:[0, T] \rightarrow V$, equipped with the norm
\begin{equation}\label{2.15}
\|u(\cdot,t)\|_{C([0, T] ; V)}=\max _{0 \leq t \leq T}\|u(\cdot,t)\|_V,    
\end{equation}
forms a Banach space. 

The symbol $C^1([0, T]; V)$ denotes the Banach space of functions whose derivative exists and belongs to $C([0, T]; V)$, endowed with the norm
\begin{equation}\label{2.16}
  \|u(\cdot,t)\|_{C^1([0, T] ; V)}=\|u(\cdot,t)\|_{C([0, T] ; V)}+\|{u_t(\cdot,t)}\|_{C([0, T] ; V)}.  
\end{equation}

\section{Dependence of the boundary conditions on the degree of degeneration}

In this section, we study the dependence of the boundary conditions on the degree of degeneration. For this aim, let us consider 
the following equation 
\begin{equation}\label{eq2.1}
    Av\equiv -\frac{d}{dx}\left(x^{\beta} \frac{dv}{dx}\right)=f(x),
\end{equation}
where $f(x)\in L^2(0,1)$ is given function, $\beta \in \mathbb{R}$ such that $0<\beta<2$ and $\beta\ne 1$.

The domain $D(A)$ of the operator $A$ consists of functions satisfying the following requirements:
\begin{equation*}
\begin{aligned}
a) & \quad  v(x) \in C[0,1], \text{the function} \quad x^{\beta}v^{\prime}(x) \quad \text{has a continuous derivative on} \quad [0,1] \\ & \quad \text{and}  \quad \frac{d}{dx}(x^{\beta} v^{\prime}(x)) \in AC[0,1];\\
b) & \quad Av\in L^2(0,1);\\
c) & \quad \text{satisfies the boundary condition} \\
\end{aligned}
\end{equation*}
    \begin{equation}\label{eq2.2}
       { v(1)=0}.
    \end{equation}

Let us determine the conditions that need to be imposed at $x=0$ for a unique solution to the given problem $\{\eqref{eq2.1},\eqref{eq2.2}\}$
 to exist.

For this aim, let us consider the following  scalar product
$$
\left(-\frac{d}{d x}\left(x^{\beta}\frac{dv}{d x}\right), v\right)=-\int_0^1 \frac{d}{d x}\left(x^{\beta}\frac{dv}{dx}\right) v dx.
$$

Hence, applying the rule of integration by parts, we obtain
\begin{equation}\label{eq2.3}\left(-\frac{d}{dx}\left(x^{\beta}\frac{d v}{d x}\right), v \right)=-\left.v \left(x^{\beta}\frac{d v}{d x}\right)\right|_0 ^1+\int_0^1 x^{\beta}\left(\frac{d v}{dx}\right)^2 dx.    
\end{equation}

We will define the conditions under which the first term on the right-hand side of \eqref{eq2.3} vanishes. 

Let $1<\beta<2$. We will show that for these values of $\beta$, the following equality
\begin{equation}\label{eq2.5}
    \lim _{x \rightarrow 0}x^\beta v^{\prime}(x)=0 
\end{equation}
is valid. 

We assume that converse, i.e. there exists some non-zero constant $b$ such that 
$$
\lim _{x \rightarrow 0} x^\beta v^{\prime}(x)=b.
$$

For instance, let $b>0$. Then for $0<p<b$ and sufficiently small $x$ the following inequality is valid: 
$$v^{\prime}(x)>\frac{p}{x^\beta}.$$

Due to the last inequality the following integral
$$\int_0^x v^{\prime}(x) d x$$
is divergent, contradicting the continuity of $v(x)$. From this contradiction, 
we get the proof of equality \eqref{eq2.5}.

Thus, we proved that if $0<\beta<2$  we do not need any conditions at $x=0$ for vanishing the first term on the right-hand side of  \eqref{eq2.3}.

Now, we consider the case $0 < \beta <1$.  In this case, we have 
$$
\lim _{x \rightarrow 0}x^{\beta} v^{\prime} \neq 0.
$$

Thus, we proved that for any function $v\in D(A)$, in case $0<\beta<1$, the condition $v(0)=0$ is
a necessary condition for vanishing the first term on the right-hand side of \eqref{eq2.3} and in case $\beta>1$, we do not need any conditions at $x=0$.  

\section{Main Problem Outline and Exploration}

For equation \eqref{eq3.1}, we study the following boundary value problem: 

\begin{problem}\label{problem4.1}
 Find a solution in the domain $\Omega$ of the equation \eqref{eq3.1} satisfying the following initial
\begin{equation}\label{eq3.2}
  u(x, a+)=\varphi(x), \quad x\in\left[0, 1\right]
\end{equation}
 and the boundary conditions: 
\begin{equation}\label{eq3.3}
\frac{\partial^\nu u(0, t)}{\partial x^\nu}=0; \quad \nu=\overline{0, -[\beta]}; \quad u(1,t)=0, \quad a\leq t\leq T
\end{equation}
where   $\varphi$ is given function, $[\beta]$ denotes the integer part of $\beta$,  $a \geq 0$.
\end{problem}

First, we will examine the form of the boundary condition \eqref{eq3.3} for different values of $\beta$.

Let $\beta\in (0,1)$. Then, the condition \eqref{eq3.3} takes the form
\begin{equation}\label{eq3.4}
u(0,t)=0, \quad  u(1,t)=0, \quad a \leq t \leq T.    
\end{equation}
 
For the values of $\beta\in (1,2)$, the condition \eqref{eq3.3} has the form 
\begin{equation}\label{eq3.5} u(1,t)=0, \quad a\leq t\leq T,    
\end{equation}
so we do not need any conditions on the line $x=0$.

Therefore, the condition \eqref{eq3.3} encompasses the conditions \eqref{eq3.4}, \eqref{eq3.5}.

\section{Spectral problem}

For studying problem \ref{problem4.1}  we use the method of separation variables, i.e. we will seek the solution of the equation \eqref{eq3.1} satisfying the boundary conditions \eqref{eq3.3} in the form 
$$u(x,t)=T(t)v(x),$$
where $T(t)$ and $v(x)$ are unknown functions. 

Then, we obtain the following spectral problem for $v(x)$:
\begin{equation}\label{eq4.1}
-\left( x^{\beta} v^{\prime}(x)\right)^{\prime}=\lambda v(x),\end{equation}

\begin{equation}\label{eq4.2}
v^{(\mu)}(0)=0, \quad \mu=\overline{0,-[\beta]}, \quad v(1)=0. \end{equation}

For proving the existence of the eigenvalues and eigenfunctions of the spectral problem $\{\eqref{eq4.1},\eqref{eq4.2}\}$ we consider the following operator 
$$
A v \equiv -\left( x^{\beta}v^{\prime}\right)^{\prime}
$$ with the domain $D(A)$ that consists of functions with the following properties: 

1) $v \in C[0,1];$

2) $x^\beta v^{\prime}(x) \in C[0,1]$ and  $(x^\beta v^{\prime}(x))^{\prime}\in AC[0,1]$,  
$Av \in L^2(0,1)$;

3) $v^{(\mu)}(0)=0, \quad \mu=\overline{0,-[\beta]}, \quad v(1)=0.$

 We denote by $\dot{W}_{2, \beta}^1$ the closure of the set $D(A)$ with the following norm
$$
\|v\|_{W_{2, \beta}^1}^2=\int_0^1\left[v^2(x)+x^\beta v^{{\prime}2}(x)\right] dx.
$$

For proving the existence of the eigenfunctions and eigenvalues of the spectral problem $\{\eqref{eq4.1},\eqref{eq4.2}\}$ we use the variational method \cite{M64}. To do this we show the positive definite of the operator $A$.  

\begin{theorem}\label{th4.1}
   Let  $0 < \beta < 2$ and let $\beta \ne 1$. Then the operator $A$ is positive definite in $L^2(0,1)$.
\end{theorem}

\begin{proof}
    First, we will prove that the operator $A$ is symmetric. Since the domain of the definition of the operator contains finite functions on $(0,1)$, the domain of the operator $D(A)$ is dense in $L^2(0,1)$. 
    
    Now, for any functions $v, \omega \in D(A)$, we consider the following scalar product 
$$
(A v, \omega)=\int_0^1 \left(x^{\beta} v^{\prime}\right)^{\prime} \omega d x.
$$

Applying the rule of integration by parts, and considering the conditions \eqref{eq4.2} from the last, we obtain 
\begin{equation}
(A v, \omega)=\int_0^1 x^{\beta} v^{\prime} \omega^{\prime} dx.
\end{equation}

Hence, by applying integration by parts once again and taking \eqref{eq4.2} into account, we get
\begin{equation}\label{eq4.4}
(A v, \omega)=\int_0^1 x^{\beta}  \omega^{\prime}   v^{\prime} dx=(v,A\omega).    
\end{equation}

Thus, we obtain the equality $(A v, \omega) = (v, A\omega)$, from which it follows that the operator $A$ is symmetric.

Now, we shall show that for any  $v \in D(A)$, there exists a positive constant $\gamma$, and the operator $A$ satisfies the following positive definiteness inequality 
\begin{equation}\label{eq4.5}
(A v, v) \geqslant \gamma \|v\|_{L^2(0,1)}^2.
\end{equation}

From the equality   (\ref{eq4.4}), in case $w=v$, we obtain  
\begin{equation}\label{eq4.6}
(A v, v) = \int_0^1 x^{\beta} v^{\prime 2}(x) dx.
\end{equation}

Since $v(1)=0$, we can write the following equality
$$v\left(x \right)=-\int\limits_{x}^{1}{{v}^{\prime}\left( t \right)dt}.$$

 Using this equality, we have
$$\int\limits_{0}^{1}v^{2}\left( x \right)dx=\int\limits_{0}^{1}\left(\int\limits_{x}^{1}{{{v}^{\prime}}\left( t \right)dt}\right)^{2}dx=\int\limits_{0}^{1}\left(\int\limits_{x}^{1}{\frac{{{1}}}{\sqrt{t^\beta}}\sqrt{t^{\beta}}{{v}^{\prime}}\left( t \right)dt}\right)^{2}dx.$$

Hence, applying the Cauchy-Schwarz inequality, we obtain
$$\int\limits_{0}^{1}{{{v}^{2}}}\left( x \right)dx\le \int\limits_{0}^{1}{\left( \int\limits_{x}^{1}{{{t^{-\beta}}}dt} \right)\left( \int\limits_{x}^{1}{{t^\beta}{{\left[ {{v}^{\prime}}\left( t \right) \right]}^{2}}dt} \right)}\,dx\le$$
\begin{equation}\label{2.6}
\le \int\limits_{0}^{1}{\left( \int\limits_{x}^{1}{{{t^{-\beta}dt}}} \right)}\,dx\left( \int\limits_{0}^{1}{{t^\beta}{{\left[ {{v}^{\prime}}\left( t \right) \right]}^{2}}dt} \right). \end{equation}

From the last, we have 
$$
\int\limits_0^1 v^2(x) d x \leqslant \int\limits_0^1 x^\beta v^{\prime 2}(x) d x \int\limits_0^1 \int\limits_x^1 {t^{-\beta}} dt dx.
$$

Since $\beta \ne 1$, we have 
$$
\int\limits_0^1 \int\limits_x^1 {t^{-\beta}} dt dx=\frac{1}{1-\beta}-\frac{1}{(1-\beta)(2-\beta)}=\frac{1}{2-\beta}>0.
$$

Then, introducing 
 $$\gamma={2-\beta}$$
 we come the proof of the inequality \eqref{eq4.5}.

Thus, the proof of Theorem \ref{th4.1} is complete.
\end{proof}

Now, we consider the Friedrichs extension of the operator $A$, and denote this extension by the same notation, $A$.

\begin{theorem}
    Let  $0<\beta<2$ and $\beta \ne 1$.  Then, the operator $A$ has a discrete spectrum. 
\end{theorem}

\begin{proof}
    Since, the operator $A$ is positive definite, we introduce the energetic space $H_A$ of the operator $A$ with the following norm
$${{\left\| v \right\|}^2_{{{H}_{A}}}}=\int\limits_{0}^{1}{{x^\beta{{{{v}^{\prime 2}}\left( x \right) }}}dx}.$$

It is easy to show that $H_A$, as a set of functions, coincides with the space
$\dot{W}_{2, \beta}^{1}(0,1)$. Moreover, the norms of these two spaces are equivalent.

Let $M$ be a set of functions $v$ for which the  norm in the energetic space $H_A$ are bounded, i.e ${{\left\| v \right\|}_{{{H}_{A}}}} \leqslant c$, where $c$ is a finite constant number.

Then, from 
\eqref{eq4.4}  and  \eqref{eq4.5} it follows that 
$$
\|v\|_{W_{2, \beta}^2}^1 \equiv \int_0^1\left[v^2(x)+x^\beta v^{\prime 2}(x)\right] d x \leqslant \text { const.}
$$

Then, by the Kondrashov embedding theorems for weighted classes \cite{K62}, the set $M$ is compact in the spaces into which it is embedded, specifically:

a) in the spaces of continuous functions if $0<\beta<1$;

b) in the space $L^q$ if $\beta>1$, where $q \in \mathbb{R}$ such that $q<2 /(\beta-1)$.

 These statements show that the set $M$  is compact in the space  $L^2(0,1)$. 
Then, based on Theorem 3 (\S 40) in \cite{M64}, we conclude that the spectrum of the operator $A$ is discrete, i.e.  the system of eigenfunctions of the operator $A$ is a complete orthonormal system in $L^2(0,1)$ and orthogonal in $H_A$:
\begin{equation}\label{eq4.12}
 \int_0^1 v_i(x) v_j(x) d x=\left\{\begin{array}{lll}
1, & \text { if } & i=j ; \\
0, & \text { if } & i \neq j,
\end{array}\right. \\
\end{equation}
\begin{equation}\label{eq4.13}
 \int_0^1 x^\beta v_i^{\prime}(x) v_j^{\prime}(x) d x= \begin{cases}\lambda_i, \text { if } & i=j, \\
0, \text { if } & i \neq j .\end{cases}
\end{equation}
where $\left\{v_n\right\}_{n=1}^{+\infty}$  and  $\left\{\lambda_n\right\}_{n=1}^{+\infty}$  are the eigenfunctions and eigenvalues of the spectral problem  \{\eqref{eq4.1}, \eqref{eq4.2}\}, respectively.
\end{proof}

\section {The order of Fourier coefficients}
In this section, we establish the convergence of some series that will be used throughout the paper.

Let $\left\{v_n\right\}_{n=1}^{+\infty}$  and  $\left\{\lambda_n\right\}_{n=1}^{+\infty}$  be the eigenfunctions and eigenvalues of the spectral problem  \{\eqref{eq4.1}, \eqref{eq4.2}\}.

\begin{lemma}\label{Lemma3}
    Let  $f \in \dot{W}_{2, \beta}^{1}(0,1)$. Then, the following inequality holds:
\begin{equation}\label{eq5.1}
\sum_{n=1}^{+\infty} \lambda_n f_n^2 \leqslant \int_0^1 x^\beta \left[f^{\prime}(x)\right]^2 d x\end{equation}
from which follows the convergence of the series on the left-hand side of \eqref{eq5.1}, where 
$f_n=\int\limits_0^{1} f(x)v_n(x) dx$.
\end{lemma}

\begin{proof}

By the definition of generalized eigenfunctions for any    $f \in \dot{W}_{2, \beta}^{1}(0,1)$ the following equality holds
    \begin{equation}\label{eq5.2}
\int_0^1  x^\beta v_n^{\prime}(x) f^{\prime}(x) d x=\lambda_n f_n.
\end{equation}

Let us consider the following non-negative expression:
$$
\int_0^1  x^\beta\left(f^{\prime}(x)-\sum_{i=1}^n f_i v_i^{\prime}(x)\right)^2 d x \geqslant 0.
$$

From the last, we obtain
\begin{equation*}
    \begin{aligned}
& \int_0^1 x^\beta f^{\prime 2}(x) d x-2 \sum_{i=1}^n f_n \int_0^1 x^\beta f^{\prime }(x) v_n^{\prime}(x) d x+ \\
+ & \sum_{i=1}^n f_n^2 \int_0^1 x^\beta v_n^{\prime 2}(x) d x+2 \sum_{\substack{i, j=1 \\
i \ne j}}^n f_i f_j \int_0^1 x^\beta v_i^{\prime}(x) v_j^{\prime}(x) d x \geqslant 0.
\end{aligned}
\end{equation*}

Considering the last equality and \eqref{eq4.13}, \eqref{eq5.2}, from the last we have 
$$
\sum_{i=1}^{n} \lambda_i f_i^2 \leqslant \int_0^1 x^\beta \left[f^{\prime}(x)\right]^2 d x.
$$

Hence, passing to the limit as $n \to +\infty$ we come to the inequality \eqref{eq5.2}.
\end{proof}

Let us introduce the following notation 
$$A^mf=	\underbrace{A(A(...(A}_{m\,\, \text{times}} f)...)).
		$$

~ 

~

\begin{lemma}\label{lemma5.2}
Let the following conditions be given
\begin{enumerate}
    \item for all even $m$ numbers $f,$ $Af$, $\dots$, $A^{\frac{m}{2}}f\in \dot{W}_{2, \beta}^{1}(0,1)$;
    \item for all odd $m$ numbers $f,$ $Af$, $\dots$, $A^{\frac{m-1}{2}}f\in \dot{W}_{2, \beta}^{1}(0,1)$, $A^{\frac{m+1}{2}}f\in L^2(0,1)$.
\end{enumerate}     

Then, the following Bessel-type inequalities hold true:
\begin{equation}\label{eq5.3}
    \begin{aligned}
\sum_{n=1}^{+\infty} \lambda_n^{m+1} f_n^2 \leqslant\left\{\begin{array}{l}
\int_0^1 x^\beta\left[\left(A^{\frac{m}{2}} f\right)^{\prime}\right]^2 d x, \quad \text {for  even $m$, } \\
\int_0^1\left[A^{\frac{m+1}{2}} f\right]^2 d x, \quad  \text {for odd $m$}
\end{array}\right.
\end{aligned}
\end{equation}
\end{lemma}

\begin{proof}
    Assume that $m$ is even natural number. Then, applying the rule of integration by parts twice, from \eqref{eq5.2} we have 
\begin{equation}\label{eq5.4}
    \begin{aligned}
\left(A f, v_n\right)=\lambda_n f_{n}=\lambda_n\left(f, v_n\right).
\end{aligned}
\end{equation}

Replacing  $f$ by $Af$ in \eqref{eq5.4}, we obtain 
    $$
    \left(A^2 f, v_n\right)=\lambda_n\left(A f, v_n\right)=\lambda_n^2 f_n.
    $$

Similarly, one can show that the following equality is valid:
\begin{equation}\label{eq5.5}
    \left(A^{\frac{m}{2}} f, v_n\right)=\lambda_n^{\frac{m}{2}} f_n.
\end{equation}

Equation \eqref{eq5.5} shows that $\lambda_n^{\frac{m}{2}} f_n$ is the Fourier coefficient of the function $A^{\frac{m}{2}} f$ for the system of the eigenfunctions of the spectral problem \{\eqref{eq4.1}, \eqref{eq4.2}\}.

Since, $A^{\frac{m}{2}} f \in \dot{W}_{2, \alpha}^{1}(0,1)$  by applying Lemma \ref{lemma5.2}, we have 
$$
\sum_{n=1}^{+\infty} \lambda_n \lambda_n^m f_n^2=\sum_{n=1}^{+\infty} \lambda_n^{m+1} f_n^2 \leqslant \int_0^1 x^\beta\left[\left(A^{\frac{m}{2}} f\right)^{\prime}\right]^2 dx.
$$

Thus, we have proved Lemma \ref{lemma5.2} for even 
$m$.

Now, we consider the case where 
$m$ is odd.

In this case, from \eqref{eq5.5}, we have 
$$
\left(A^{\frac{m+1}{2}} f, v_n\right)=\lambda_n^{\frac{m+1}{2}} f_n.
$$

Under the assumptions of Lemma 6.2, we have $A^{\frac{m+1}{2}} f \in L^2(0,1)$. Then, applying Bessel's inequality, we obtain 
$$
\sum_{n=1}^{+\infty}\left(A^{m+1} f, v_n\right)^2=\sum_{n=1}^{+\infty} \lambda_n^{m-1} f_n^2 \leqslant \int^1\left[A^{\frac{m+1}{2}} f\right]^2 d x. 
$$

The proof of Lemma \ref{lemma5.2} is complete. 
\end{proof}

\begin{lemma}\label{lemma5.3}
    Let $f \in C([a, T], \dot{W}_{2, \beta}^1(0,1))$, $A f \in C\left([a, T], L^2(0,1)\right)$. Then the following series
    $$
    \sum_{n=1}^{+\infty} \lambda_n^2 f_n^2(t)
    $$
    converges absolutely and uniformly in $[a,T]$.
\end{lemma}

\begin{proof}
By employing Parseval's equality for the function $Af(x,t)$, we have 
$$
\sum_{n=1}^{+\infty} \lambda_n^2 f_n^2(t)=\int_0^1[A f(x, t)]^2 d x.
$$

Since $A f(x, t) \in C\left([a, T], L^2(0,1)\right)$ based on Dini's theorem, we have the proof of Lemma  \ref{lemma5.3}. 
\end{proof}

\begin{lemma}\label{lemma5.4}
    Let $f(x, t), A f(x, t) \in C([a, T], \dot{W}_{2, \beta}^1(0,1))$. Then the following series
    $$
    \sum_{n=1}^{+\infty} \lambda_n^3 f_n^2(t)
    $$
    converges absolutely and uniformly in $[a,T]$.
\end{lemma}

\begin{proof}
Since  $(Af,v_n)=\lambda_nf_n(t)$ and $A f(x, t) \in C([a, T], \dot{W}_{2, \beta}^1(0,1))$, by applying Lemma \ref{Lemma3}, we have
$$
\sum_{n=1}^{+\infty} \lambda_n^3 f_n^2(t)\leq \int_0^1x^{\beta}[(A f)^{\prime}]^2 d x < +\infty.
$$

Since $A f(x, t) \in C([a, T], \dot{W}_{1, \beta}^2(0,1))$ based on Dini's theorem, we have the proof of Lemma  \ref{lemma5.4}.
\end{proof}

\section{The existence and uniqueness of the solution to Problem \ref{problem4.1}}

We prove the existence and uniqueness of the solution to Problem \ref{problem4.1}.

To do this, we consider two cases:

\textbf{Case 1.} $0<\beta<1$. 

\textbf{Case 2.} $1<\beta <2$.

\textbf{Case 1. Existence of the solution.}

First, we will define the classical solution to the problem. 

\begin{definition}
A function $u(x,t)$, such that $u(x,t) \in C(\bar{\Omega})$, \quad ${}^C \big( \left(t-a\right)^\theta \frac{\partial}{\partial t} \big)^\alpha  u(x,t), $ $ Au \in C(\Omega)$ satisfies the conditions (\ref{eq3.2}), (\ref{eq3.4}), is called a classical solution of problem \ref{problem4.1}.
\end{definition}

We will prove that the problem has a classical solution in Case 1.  

Let $\{v_k\}_{k=1}^{+\infty}$ and $\{\lambda_k\}_{k=1}^{+\infty}$ be the eigenfunctions and eigenvalues of spectral problem $\{\eqref{eq4.1},\eqref{eq4.2}\}$. We seek a solution to the problem in the following form
\begin{equation}\label{eq7.1}
u(x,t)=\sum\limits_{k=1}^{+\infty}u_k(t)v_{k}(x),
\end{equation}
where $(u,v_k)_{L^2(0,1)}=u_k(t)$ for $k\in \mathbb{N}.$

Substituting \eqref{eq7.1} into \eqref{eq3.1} 
and introducing $(f,v_k)_{L^2(0,1)}=f_k(t)$ from \eqref{eq3.1}, we obtain 
\begin{equation}\label{eq7.2}
^C\left(\big(t^{\theta}-a^\theta\big)\frac{d}{dt}\right)^{\alpha}u_k(t)+\lambda_k u_k(t)=f_k(t), \quad k \in \mathbb{N}, \quad t \in(0,T).
\end{equation}

Additionally, from the initial condition \eqref{eq4.2}, we have
\begin{equation}\label{eq7.3}
u_k(a^+)=\varphi_k,  \quad  k \in \mathbb{N},
\end{equation}
where $\varphi_k$ is the Fourier coefficient of the function $\varphi(x)$ with respect to the system of the eigenfunctions $\{v_k\}_{k=1}^{+\infty}$ can be defined as:
$$
\varphi_k=\int_0^1{\varphi(x)v_k(x)dx}.
$$

Using the result of Theorem  \ref{theo2.4}, the solution of the problem \big\{\eqref{eq7.2},\eqref{eq7.3}\big\} defined by 
\begin{equation}\label{eq7.4}
		\begin{gathered}
		u_k(t)=\varphi_k E_{\alpha, 1}\left[\lambda^*\big(t^p-a^p\big)^\alpha\right]+B_k(t),
		\end{gathered}
	\end{equation}
 where 
 \begin{multline*}
		\displaystyle{B_k(t)=\frac{1}{p^\alpha\Gamma(\alpha)}\int_{a}^{t}\big(t^p-\tau^p\big)^{\alpha-1}f_k(\tau)d(\tau^p)}\\
		\\
		\displaystyle{+\frac{\lambda^*}{p^\alpha}\int\limits_{a}^{t}\left(t^{p}-x^{p}\right)^{2\alpha-1}E_{\alpha, 2\alpha}\left[\lambda^*\left(t^{p}-x^{p}\right)^\alpha\right]f_k(x)d(x^p),}
	\end{multline*}
	
	where $p=1-\theta$ and $\lambda^*=-\frac{\lambda_k}{p^{\alpha}}$.

Substituting the obtained expression of $u_k(t)$ into \eqref{eq7.1}, we have 
\begin{equation}\label{eq7.5}
    u(x,t)=\sum\limits_{k=1}^{+\infty}\left\{\varphi_k E_{\alpha, 1}\left[\lambda^*\big(t^p-a^p\big)^\alpha\right]+B_k(t)\right\}v_{k}(x).
\end{equation}

\begin{theorem}\label{th7.1}
    Let the following conditions be fulfilled:
    \begin{enumerate}
        \item $0<\beta<1$;
        \item $\varphi, A\varphi \in \dot{W}_{2, \beta}^{1}(0,1);$
        \item $f, Af \in C ([a,T], \dot{W}_{2, \beta}^{1}(0,1))$.
    \end{enumerate}

        Then, the function $u(x,t)$ defined by  \eqref{eq7.1} will be the classical solution to Problem \ref{problem4.1}.

\begin{proof}
To prove Theorem \ref{th7.1} it is sufficient to show continuity of the functions $Au$ and ${ }^C\Big(\big(t-a)^\theta \frac{\partial}{\partial t} \Big)^\alpha u$ in the domain $\bar{\Omega}$.  If we show $Au \in C(\bar{\Omega})$, then from the equation \eqref{eq3.1} it follows the continuity of the function ${ }^C\Big(\big(t-a)^\theta \frac{\partial}{\partial t} \Big)^\alpha u$. For this aim, by formally differentiating \eqref{eq7.1}, we have 
\begin{equation}
    Au\equiv -\frac{\partial}{\partial x}\left(x^{\beta} \frac{\partial u(x,t)}{\partial x}\right)=\sum\limits_{k=1}^{+\infty} \lambda_k u_k(t)v_k(x). 
\end{equation}

Hence, considering \eqref{eq4.13}, we have 
\begin{equation}
    {{\left\| Au \right\|}^2_{{{H}_{A}}}}=\sum\limits_{k=1}^{+\infty} \lambda^3_k u^2_k(t).
\end{equation}

Taking the Lemma \ref{lemmacha} into account, the solution of the boundary value problem can be also written as 
\begin{equation*}
    \begin{aligned}
     u(x, t)=&\sum_{k=1}^{+\infty}\left[\varphi_k E_{\alpha, 1}\left(\lambda^*\left(t^p-a^p\right)^\alpha\right)\right.\\ +&\left.\frac{1}{p^\alpha} \int_a^t\left(t^p-\tau^p\right)^{\alpha-1} E_{\alpha, \alpha}\left[\lambda^* \left(t^p-\tau^p\right)^\alpha\right] f_k(\tau) d \tau^p \right] v_k(x),    
    \end{aligned}
\end{equation*}   where $\left({\lambda^*}={-\frac{\lambda_k}{p^\alpha}}\right)$

Now taking ${H}_A$-norm yields after integrating by parts
\begin{equation*}
    \begin{aligned}
 \|Au(\cdot, t)\|_{H_A}^2 \leqslant &  \sum_{k=1}^{+\infty}\lambda^3_k\left|\varphi_k\right|^2\left|E_{\alpha, 1}\left(\lambda^*\left(t^p-a^p\right)^\alpha\right)\right|^2\\ + &  \sum_{k=1}^{+\infty}\lambda^3_k \big|f_k(a)\big|^2\Big|(t^p-a^p)^\alpha E_{\alpha, \alpha+1}\big[\lambda^*(t^p-a^p)^\alpha\big]\Big|^2 \\+&
\sum_{k=1}^{+\infty} \frac{\lambda^3_k}{p^{2 \alpha}} \Big(\int_a^t\left|\left(t^p-\tau^p\right)^{\alpha} E_{\alpha, \alpha+1}\left[\lambda^*\left(t^p-\tau^p\right)^\alpha\right] f^{\prime}_k(\tau)\right| d \tau^p\Big) ^2.          
    \end{aligned}
\end{equation*}

 Hence, considering Lemma \ref{lem: 2.1} and applying Cauchy-Schwarz inequality, we obtain
\begin{equation*}
\begin{aligned}
  \|Au(\cdot, t)\|_{H_A}^2  & \leqslant   \sum_{k=1}^{+\infty} {\lambda^3_k}{}\left|\varphi_k\right|^2\left(\frac{M p^\alpha}{\left.p^\alpha+\lambda_k \mid t^p-a^p\right)^\alpha}\right)^2+\sum_{k=1}^{\infty} \frac{\lambda^3_k}{p^{2 \alpha}}  \big|f_k ( a)\big|^2\left(\frac{M p^\alpha\left|t^p-a^p\right|^\alpha}{p^\alpha+\lambda_k\left|t^p-a^p\right|^\alpha}\right)^2+ \\
& +\sum_{k=1}^{\infty} \frac{\lambda^3_k}{p^\alpha} \int_a^t\left(\frac{M p^\alpha\left|t^p-\tau^p\right|^\alpha}{p^\alpha+\lambda_k\left|t^p-\tau^p\right|^\alpha}\right)^2 d \tau \int_a^t\left|f_k^{\prime}(\tau)\right|^2 d \tau \leqslant \\
& \leqslant M_1 \sum_{k=1}^{\infty} \lambda_k^2\left|\varphi_k\right|^2\left(\frac{\sqrt{\lambda_k} p^\alpha}{p^\alpha+\lambda_k\left|t^p-a^p\right|^\alpha}\right)^2+\sum_{k=1}^{\infty}\lambda_k^2\left|f_k(a)\right|^2\left(\frac{\sqrt{\lambda_k}\left(t^p-a^p\right)^\alpha}{p^\alpha+\lambda_k\left(t^p-a^p\right)^\alpha}\right)^2+ \\
& +M_3 \sum_{k=1}^{\infty} \int_a^t\left(\frac{p^\alpha \sqrt{\lambda_k}\left|t^p-t^p\right|^\alpha}{p^\alpha+\lambda_k\left|t^p-\tau^p\right|^\alpha}\right)^2 d \tau \cdot \lambda_k^2\int_a^T\left|f_k^{\prime}(\tau)\right|^2 d \tau^p \\
& \leqslant M_1 \sum_{k=1}^{\infty}\lambda_k^2\left|\varphi_k\right|^2+M_2 \sum_{k=1}^{\infty}\lambda_k^2\left|f_k(a)\right|^2+M_3 \sum_{k=1}^{\infty}\lambda_k^2 \int_a^T\left|f_k^{\prime}(\tau)\right|^2 d \tau^p.
\end{aligned}
\end{equation*}

We have also used the inequality $(a+b+c)^2 \leq 3(a^2+b^2+c^2)$  to obtain the following estimation 
\begin{equation}
    \label{eq7.7}
{{\left\| Au \right\|}^2_{{{H}_{A}}}} \leq    C_1\sum\limits_{k=1}^{\infty} \lambda^2_k \varphi_k^2+C_2\sum_{k=1}^{\infty}\lambda_k^2\left|f_k(a)\right|^2+C_3 \sum_{k=1}^{\infty}\lambda_k^2 \int_a^T\left|f_k^{\prime}(\tau)\right|^2 d \tau^p,
\end{equation}
where $C_1$, $C_2$
and $C_3$ are constants and do not depend on $k$. 

The convergence of the two series follows from the conditions of Theorem \ref{th7.1} and Lemma \ref{lemma5.2}. The convergence of the third series follows from the conditions of Theorem \ref{th7.1} and Lemma \ref{lemma5.2}, with the subsequent application of the theorem on term-by-term integration of uniformly converging series.

   Thus, we have obtained 
   \begin{equation}
    {{\left\| Au \right\|}^2_{{{H}_{A}}}}\leq const,
\end{equation}
from which by Kondrashov embedding theorems for weighted classes \cite{K62}, we have $Au\in C[0,1]$. The continuity of $Au$ for $t$ follows from the uniform convergence of the series 
$$
\sum\limits_{k=1}^{+\infty} \lambda^3_k u_k^2(t)
$$
in $[0,T]$. Thus, $Au \in C(\bar{\Omega})$.

\end{proof}
    
\end{theorem}

\textbf{Case 2.} $1<\beta <2$. \textbf{Existence of the solution.}

In this case, first, we introduce the definition of the weak solution to Problem \ref{problem4.1}.
\begin{definition}
    A function $u(x,t)$ such that $$u(x,t)\in C([a,T],L^2(0,1))\cap C((a,T],\dot{W}_{2, \beta}^{1}(0,1)),  Au \in  C((a,T], L^2(0,1)),$$ is called a weak solution to the problem \ref{problem4.1}, if it satisfies the following relation 
    \begin{equation}\label{eq.7.10}
  {}^C \big( \left(t-a\right)^\theta \frac{\partial}{\partial t} \big)^\alpha(u,\omega)_{L^2(0,1)}-\left(x^{\beta/2}\frac{\partial u}{\partial x},x^{\beta/2}\frac{\partial \omega}{\partial x}\right)_{L^2(0,1)}= \left(f,\omega\right)_{L^2(0,1)}     
    \end{equation}
    for all $\omega \in \dot{W}_{2, \beta}^{1}(0,1)$,
    and satisfies the condition \eqref{eq3.2}.
\end{definition}

Let $\{v_k\}_{k=1}^{+\infty}$ and $\{\lambda_k\}_{k=1}^{+\infty}$ be the eigenfunctions and eigenvalues of the spectral problem $\{\eqref{eq4.1},\eqref{eq4.2}\}$. We seek a solution to the problem in the following form 

\begin{equation}\label{eq.7.11}
    u(x, t)=\sum_{k=1}^{+\infty} u_k(t) \vartheta_k(x)
\end{equation}

Substituting (\ref{eq.7.11}) into (\ref{eq.7.10}) and setting $\omega(x)$ as $v_k(x)$ in (\ref{eq.7.10}), and using the following equalities
$$
\begin{aligned}
& ^C\left((t-a)^\theta \frac{\partial}{\partial t}\right)^\alpha\left(u \cdot v_k\right)_{L^2(0,1)}= ^C\left((t-a)^\theta \frac{\partial}{\partial t}\right)^\alpha u_k(t), \\
& ^C\left(x^{\beta / 2} \frac{\partial u}{\partial x}, x^{\beta / 2} \frac{\partial v_k}{d x}\right)_{L^2(0,1)}=-\lambda_k\left(u, v_k\right)=-\lambda_k u_k(t),\\
&\left(f, v_k\right)_{L^2(0,1)}=f_k(t)
\end{aligned}
$$
and from (\ref{eq.7.10}),  we obtain
\begin{equation}\label{eq.7.12}
    ^C\left((t-a)^\theta \frac{d}{d t}\right)^\alpha u_k(t)+\lambda_k u_k(t)=f_k(t), ~ k \in \mathbb{N}, ~ t \in(a, T),
\end{equation}

Moreover, from the initial condition (\ref{eq3.2}), we have

\begin{equation}\label{eq.7.13}
    u_k\left(a^{+}\right)=\varphi_k,  ~ k \in \mathbb{N}
\end{equation}

The solution to the problem has the form (\ref{eq7.4}). Substituting this expression of $u_k(t)$ into (\ref{eq.7.11}), we have the following formal representation of the solution
\begin{equation}\label{eq.7.14}
    u(x, t)=\sum_{k=1}^{+\infty}\Big(\varphi_k E_{\alpha_1 1}\left[\lambda^*\left(t^p-a^p\right)^\alpha\right]+B_k(t)\Big) v_k(x)
\end{equation}

The following theorem is valid:

\begin{theorem}\label{teor.7.4}
    Let the following conditions be fulfilled:
    
    \begin{enumerate}
        \item $1<\beta<2$;
        \item  $\varphi \in L^2(0, 1)$
        \item $f(\cdot, a)\in L^2(0, 1)$ and $f \in C^1([a, T], L^2(0, 1))$
    \end{enumerate}
    then the unique weak solution to the Problem \ref{problem4.1} exists.
\end{theorem}

\begin{proof}

Considering Lemma \ref{lemmacha}, we rewrite (\ref{eq.7.14}) in the following form
$$
\begin{aligned}
u(x, t)= & \sum_{k=1}^{+\infty}\Big[\varphi_k E_{\alpha, 1}\left(\lambda^*\left(t^p-a^p\right)^\alpha\right)+\\
 + &\frac{1}{p^\alpha} \int_a^t\left(t^p-\tau^p\right)^{\alpha-1} E_{\alpha, \alpha}\left[\lambda^*\left(t^p-\tau^p\right)^\alpha\right] f_k(\tau) d \tau^p\Big] v_k(x) .
\end{aligned}
$$

Since, $\left\{v_k(x)\right\}_{k=1}^{+\infty}$ is an orthonormal system in $L^2(0,1)$, and after applying the rule of integration by parts to the solution, we take $L^2(0, 1)$-norm

$$
\begin{aligned}
 \|A u\|_{L^2(0,1)}& =\sum_{k=0}^{\infty}\left|\lambda_k\right|^2\left|u_k(t)\right|^2 \leqslant \\
& \leqslant 3 \sum_{k=0}^{+\infty}\left|\varphi_k\right|^2 |\lambda_k|^2 \big|E_{\alpha, 1}\left(\lambda^*\left(t^p-a^p\right)^\alpha\right)\big|^2+\\&+3 \sum_{k=0}^{+\infty}\left|f_k(a)\right|^2 |\lambda_k|^2\Big|(t^p-a^p)^{\alpha} E_{\alpha, \alpha+1}\big(\lambda^*\left(t^p-a^p\right)^\alpha\big)\Big|^2+ \\
& +3 \sum_{k=0}^{\infty} \int_a^t \lambda_k^2 \Big|\left(t^p-e^p\right)^\alpha E_{\alpha, \alpha+1}\left[\lambda^*\left(t^p-\tau^p\right)^\alpha\right]\Big|^2\Big|f_k(\tau)\Big|^2 d \tau^p \leqslant \\
& \leqslant M_1 \sum_{k=0}^{\infty}\left|\varphi_k\right|^2 \Big|\frac{p^\alpha \lambda_k}{p^\alpha+\lambda_k\left|t^p-a^p\right|^\alpha}\Big|^2+M_2 \sum_{k=0}^{\infty}\left|f_k(a)\right|^2\left|\frac{p^\alpha \lambda_k\left|t^p-a^p\right|^\alpha}{p^\alpha+\lambda_k\left|t^p-a^p\right|^\alpha}\right|^2 + \\
& +M_3 \sum_{k=0}^{\infty} \int_a^t\left|f_k^{\prime}(\tau)\right|^2\left|\frac{p^\alpha \lambda_k\left|t^p-\tau^p\right|^\alpha}{p^\alpha+\lambda_k\left|t^p-\tau^p\right|^\alpha}\right|^2 d\tau^p \\
& \leqslant M_1 \sum_{k=0}^{\infty}\left|\varphi_k\right|^2+M_2 \sum_{k=0}^{\infty}\left|f_k(a)\right|^2+M_3 \sum_{k=1}^{\infty} \int_a^T\left|f_k^{\prime}(\tau)\right|^2 d \tau^p.,
\end{aligned}
$$
where $M_1$, $M_2$ and $M_3$ are constants.

If the conditions of the theorem \ref{teor.7.4} are fulfilled and considering Bessel's inequality, all three series converge uniformly.
The conditions for uniform converging a series $Au$ ensure that the other series $u$ or $^C ((t^\theta-a^\theta)\frac{\partial}{\partial t})^\alpha u$  also converge uniformly.

\end{proof}

\section{Uniqueness of solution}

To establish the uniqueness of the solution to the problem, we assume that there exist two distinct solutions, $u_1(x, t)$ and $u_2(x, t)$, that satisfy the problem defined by (\ref{eq3.1}), (\ref{eq3.2}) and (\ref{eq3.3}). Our goal is to demonstrate that $u(x, t) \equiv u_1(x, t) - u_2(x, t) \equiv 0$. For the solution $u(x,t)$, this leads to the boundary value problem:
\begin{equation*}
\begin{aligned}
\begin{cases}
  {}^C \big( \left(t-a\right)^\theta \frac{\partial}{\partial t} \big)^\alpha  u(x,t) 
- \frac{\partial}{\partial x} \left( x^\beta \frac{\partial u(x,t)}{\partial x} \right) = 0, ~ 
 a \leq t \leq T, \; ~ x \in [0,1], \\[1em]

u(x, a^+) = 0, ~ ~ ~
 x \in [0,1], \\[1em]

\frac{\partial^\nu u(0, t)}{\partial x^\nu} = 0, ~ ~ 
 \nu = \overline{0, -[\beta]}, ~ \;u(1, t) = 0,  ~ ~  a \leq t \leq T, 
\end{cases}
\end{aligned}
\end{equation*}

Let $u_k(t) = (u(x, t), v_k)$ be the Fourier coefficients of $u(x, t)$. Then, due to the self-adjointness of the operator $A$, we obtain:
\begin{equation*}
{}^C \big( \left(t-a\right)^\theta \frac{d}{dt} \big)^\alpha u_k(t) = \big({}^C \big( \left(t-a\right)^\theta \frac{d}{dt} \big)^\alpha  u, v_k\big) = -\big(Au, v_k\big) 
= -\big(u, Av_k\big)
= -\lambda_k u_k(t).
\end{equation*}

For $u_k(t)$, where $t > a$, the corresponding non-local boundary value problem is given by:
\begin{equation}\label{eq8.1}
{}^C \big( \left(t-a\right)^\theta \frac{d}{dt} \big)^\alpha u_k(t) + \lambda_k u_k(t) =0, \quad t > a, \quad u_k(a+) = 0, \quad \forall k \geq 1,
\end{equation}

The unique solution to the boundary value problem (\ref{eq8.1}) is given by 
$u_k(t)=0$ (see Formula (\ref{eq2.6})) and due to the completeness of the set of eigenvectors in $L^2(0, 1)$ we obtain $u(x, t)\equiv 0$.

\section*{Acknowledgements} 
The authors would like to express their gratitude to Prof. Erkinjon Karimov for his valuable suggestions that greatly enhanced the quality of the article

\end{document}